\documentclass[12pt]{amsproc}
\usepackage{amssymb}
\usepackage{amsmath}
\usepackage{amsfonts}
\usepackage{geometry}
\usepackage{graphicx}
\providecommand{\U}[1]{\protect\rule{.1in}{.1in}}
\theoremstyle{plain}

\newtheorem{corollary}{Corollary}

\newtheorem{lemma}{Lemma}

\newtheorem{proposition}{Proposition}

\newtheorem{theorem}{Theorem}
\numberwithin{equation}{section}

\begin{document}
\title[A lower bound for the nodal sets of Steklov eigenfunctions]{A lower bound for the nodal sets of Steklov eigenfunctions}
\author{Xing Wang and Jiuyi Zhu}
\address{
 Department of Mathematics\\
Johns Hopkins University\\
Baltimore, MD 21218, USA\\
Emails:  xwang@math.jhu.edu, jzhu43@math.jhu.edu}
\thanks{\noindent }
\date{}
\dedicatory{ }

\begin{abstract}
We consider the lower bound of nodal sets of Steklov eigenfunctions
on smooth Riemannian manifolds with boundary--the eigenfunctions of
the Dirichlet-to-Neumann map. Let $N_\lambda$ be its nodal set.
Assume that zero is a regular value of Steklov eigenfunctions. We
show that
$$H^{n-1}(N_\lambda)\geq C\lambda^{\frac{3-n}{2}}$$
for some positive constant $C$ depending only on the manifold.
\end{abstract}

\maketitle
\section{Introduction}
In this paper, we consider the lower bound estimates for nodal sets
$$N_\lambda=\{ x\in\mathcal{M}|\phi_\lambda=0\}$$
of the Steklov eigenfunctions on a smooth Riemannian manifold
$(\mathcal{N}, h)$ with boundary $(\mathcal{M}, g)$, where $dim
\mathcal{N}=n+1$ and $h|_{\mathcal{M}}=g$. The Steklov eigenvalue
problem is formulated as
\begin{equation}
\left\{
\begin{array}{lll}\triangle_h \phi_\lambda(x)=0, \quad x\in\mathcal{N}, \nonumber
\medskip \\
\frac{\partial \phi_\lambda}{\partial
\nu}(x)=\lambda\phi_\lambda(x), \quad x\in
\partial\mathcal{N}=\mathcal{M}. \nonumber
\end{array}
\right.
\end{equation}
Here, $\nu$ is an unit outer normal vector on $\mathcal{M}$.  The
Steklov eigenvalues can also be reduced to the boundary
$\mathcal{M}$. Then the $\phi_\lambda$ becomes the eigenfunction of
Dirichlet-to-Neumann operator, i.e.
$$\Lambda\phi_\lambda=\lambda\phi_\lambda.         $$
 The Dirichlet-to-Neumann operator $\Lambda$ is
defined as
\begin{equation}
\Lambda f=\frac{\partial}{\partial \nu}(Hf)|_{\mathcal{M}} \nonumber
\end{equation}
for $f\in H^{\frac{1}{2}}(\mathcal{M})$. $Hf$ is the harmonic
extension of $f$, i.e.
\begin{equation}
\left \{\begin{array}{lll}
\triangle_h u(x)=0, &\quad x\in \mathcal{N}, \nonumber \medskip\\
u(x)=f(x), &\quad x\in\partial\mathcal{N}=\mathcal{M} \nonumber
\end{array}
\right.
\end{equation}
with $u=Hf$. Moreover, the operator $\Lambda$ is a self-adjoint
operator from $H^{\frac{1}{2}}(\mathcal{M})$ to $
H^{-\frac{1}{2}}(\mathcal{M})$ and there exists an orthonormal basis
$\{ \phi_j\}$ of eigenfunctions such that
$$\Lambda \phi_j=\lambda_j \phi_j, \quad \phi_j\in C^\infty(\mathcal{M}), \quad
\int_{\mathcal{M}}\phi_j\phi_k\, dV_g=\delta_{jk}.    $$ The
eigenvalues $ 0=\lambda_0<\lambda_1\leq \lambda_2\leq
\lambda_3,\cdots, $ are ordered in ascending order with counted
multiplicity. For simplicity, we choose $n+1$ as the dimension of
$\mathcal{N}$, which is a little bit different from the previous
work by \cite{BL} and \cite{Zel}.

The nodal sets are zero level sets of eigenfunctions. We want to
study the asymptotical behavior of the size of nodal sets of Steklov
eigenfunctions for large $\lambda$. Recently, some remarkable
progresses have been made for the upper bound of the size of nodal
sets for analytic manifolds. Bellova and Lin \cite{BL} proved that
if $\mathcal{N}$ is an analytic domain in $\mathbb R^n$, then the
$H^{n-2}$-Hausdorff measure of nodal sets of Steklov eigenfunctions
has an upper bound of $C\lambda^6$ with $C$ depending only on
$\mathcal{N}$. Later on, Zelditch \cite{Zel} improved their results
and showed that the optimal upper bound for the nodal sets is
$C\lambda$ for real analytic manifolds. The optimality can be seen
from the case that the manifold is a ball.

So far, nothing seems to be known for the lower bound of the nodal
sets of Steklov eigenfunctions, even for analytic manifolds. The
main goal of our paper is to address the lower bound of nodal sets
over general compact smooth manifolds. Quite different from the case
for the Laplacian-Beltrami operator, the Dirichlet-to-Neumann
operator is a non-local operator, which causes additional
difficulty. Fortunately, since we are measuring the whole size of
the nodal sets which can be considered as ``partial global"
quantity, we are able to find a way to overcome the difficulty and
carry the argument through.

Let's first briefly review the literature concerning the nodal sets
of classical eigenfunctions. Let $\phi_\lambda$ be an $L^2$
normalized eigenfunctions of Laplacian-Beltrami on compact manifold
$(\mathcal{M}, g)$ without boundary,
$$ -\triangle_g \phi_\lambda=\lambda^2\phi_\lambda.                $$
Yau conjectured that for any smooth manifold, one should control the
upper and lower bound of nodal sets of classical eigenfunctions as
$$c\lambda\leq H^{n-1}( N_\lambda)\leq C\lambda       $$
where $C, c$ depends only on the manifold $\mathcal{M}$. The
conjecture is only verified for real analytic manifold by
Donnelly-Fefferman in \cite{DF}. For the smooth manifolds, the
conjecture is still not settled. Much progresses have been obtained
towards the lower bound of nodal sets. Colding and Minicozzi
\cite{CM}, Sogge and Zelditch \cite{SZ}, \cite{SZ1} independently
obtained that
$$ H^{n-1}(N_\lambda)\geq C\lambda^{\frac{3-n}{2}}      $$
for smooth manifolds. See also \cite{HS} for deriving the same bound
by adapting the idea in \cite{SZ}. For other related works about
lower bounds of nodal sets of classical eigenfunctions, see
\cite{Br}, \cite{M}, \cite{HL},  etc, to just mention a few. The
methods in \cite{CM} and \cite{SZ} are quite different. Specially,
the method in \cite{SZ} is based on a Dong-type identity in \cite{D}
about $L^1$ norm of $|\nabla \phi_\lambda|$ on the nodal set and the
$L^1$ norm of $\phi_\lambda$ on $\mathcal{M}$. Our goal is to adapt
their idea to the setting of non-local operator, i.e. Steklov
eigenfunctions.

\begin{theorem}
Let $\phi_\lambda$ be a normalized Steklov eigenfunction and $0$ be
a regular value of $\phi_\lambda$. Then there exists $C$ depending
only on $\mathcal{N}$ such that $$ H^{n-1}(N_\lambda)\geq
C\lambda^{\frac{3-n}{2}}.$$  \label{th1}
\end{theorem}

With a small modification of the proof, we are also able to get
similar lower bounds for general level sets near 0. Denote
$\alpha$-level sets of Steklov eigenfunctions as
$L^\alpha_\lambda=\{ x\in\mathcal{M}|\phi_\lambda=\alpha\}.$

\begin{corollary}

Let $\alpha$ be a regular value of $\phi_\lambda$. There exists a
positive constant $\epsilon(\mathcal{N})$ such that, for
$|\alpha|<\epsilon(\mathcal{N})\lambda^{-\frac{n-1}{4}}$,
$$H^{n-1}(L^\alpha_\lambda)\geq C\lambda^{\frac{3-n}{2}}$$
with $C$ depending only on $\mathcal{N}$.

\end{corollary}

{\bf Acknowledgements:} It is our pleasure to thank Professor
Christopher D. Sogge for many fruitful discussions throughout the
preparation of this work. We appreciate his insightful and useful
comments, which helped to improve this paper much. We also thank
referees for constructive comments.

\section{Preliminaries}
In this section, we will review and prepare some general results
needed in the proof of Theorem 1. First, we need the following
result from \cite{T}.
\begin{lemma}
The Dirichlet-to-Neumann operator $\Lambda$ is an elliptic
self-adjoint pseudodifferential operator of order $1$ over
$\mathcal{M}.$ Moreover,
$$ \Lambda=\sqrt{-\triangle_g} \ \mbox{mod} \ OPS^0(\mathcal{M}).$$
\label{lem1}
\end{lemma}
Here, $OPS^m$ denotes the pseudodifferential operator of order $m$.
Since $\Lambda$ is an elliptic self-adjoint pseudodifferential
operator, by the general results in \cite{SS} (see also the book of
 Sogge \cite{S} or \cite{S1} for Laplacian-Beltrami operator), we have the following
$L^p$ norm estimates.
\begin{lemma}
Let $\phi_\lambda$ be the Steklov eigenfunction. One has the
estimates, for $p\geq 2$,
\begin{equation}
\|\phi_\lambda\|_{L^p(\mathcal{M})}\lesssim
(1+\lambda)^{\sigma(n,p)}\|\phi_\lambda\|_{L^2(\mathcal{M})},
\label{sog}
\end{equation} where
\begin{equation}
\sigma(n, p)= \left\{ \begin{array}{lll} \
n(\frac{1}{2}-\frac{1}{p})-\frac{1}{2}, \quad \frac{2(n+1)}{n-1}\leq
p\leq \infty, \nonumber \medskip \\
\frac{n-1}{2}(\frac{1}{2}-\frac{1}{p}), \quad 2\leq p \leq
\frac{2(n+1)}{n-1}.
\end{array}
\right.
\end{equation}
\label{lem2}
\end{lemma}
In the whole paper, the notation $A\lesssim B$ or $ A\gtrsim B$
denotes $A\leq CB$ or $A\geq CB$ for some generic constant $C$ which
does not depend on $\lambda$. If we follow exactly the same argument
as \cite{SZ}, which makes use of lemma \ref{lem2} for $p=\infty$, we
can obtain $L^p$ norm estimates for $p=1$, that is,
\begin{equation}
\|\phi_\lambda\|_{L^1(\mathcal{M})}\gtrsim
(1+\lambda)^{-\frac{n-1}{4}}\|\phi_\lambda\|_{L^2(\mathcal{M})}.
\label{less}
\end{equation}

We also need the $L^p$ bounds for the pseudodifferential operators.
\begin{lemma}
Suppose $P\in OPS^m(\mathcal{M})$. Then
$$\|P\phi_\lambda\|_{L^{p}(\mathcal{M})}\lesssim(1+\lambda)^m\|\phi_\lambda\|_{L^{p}(\mathcal{M})}, \quad
\forall  1<p<\infty. $$ Specially,
$$\|\nabla^m_g\phi_\lambda\|_{L^{p}(\mathcal{M})}\lesssim (1+\lambda)^{m}
\|\phi_\lambda\|_{L^{p}(\mathcal{M})}. $$  \label{lem3}
\end{lemma}
\begin{proof}
Define the operator $\tilde{P}:=P(1+\Lambda)^{-m}$, then
$\tilde{P}\in OPS^0(\mathcal{M})$. By the boundedness of zeroth
pseudodifferential operator over $L^p(\mathcal{M})$ in \cite{S} or
\cite{T}, the lemma follows easily.
\end{proof}

\section{Lower bounds of nodal sets}
In this section, we will obtain the lower bounds of nodal sets of
Steklov eigenfunctions. Since the Dirichlet-to-Neumann operator is a
non-local operator, we
 do not need  information from the manifold $(\mathcal{N},
h)$. In the following argument, all derivatives and calculations are
performed with respect to the manifold $(\mathcal{M}, g)$. We first
express the manifold $\mathcal{M}$ as the disjoint union
$$\mathcal{M}=\bigcup_{j=1}^{N_+(\lambda)} D_{j, +}
\cup \bigcup_{j=1}^{N_-(\lambda)} D_{j,-} \cup N_{\lambda},
$$
where $D_{j, +}$ and $D_{j, -}$ are the connected components of the
sets $\{x\in\mathcal{M}|\phi_\lambda>0\}$ and
$\{x\in\mathcal{M}|\phi_\lambda<0\}$. Using the same idea in
\cite{SZ}, we can treat each component separately and then add them
up. For simplicity, we just deal with two components. The same
argument carries out for many components.
 Denote
$$ D_{+}=\{x\in \mathcal{M}|\phi_\lambda(x)>0\} $$
and $$D_{-}=\{x\in \mathcal{M}|\phi_\lambda(x)<0\}.$$

 For the classical eigenfunctions of Laplacian-Beltrami operator, the singular set is codimension
 2. Then zero level sets are smooth submanifolds. It is also shown in \cite{U} that $0$ is regular
 value for eigenfunctions of second order elliptic differential
 operators. To the best of the authors' knowledge, it is still
 unknown whether it is true for Dirichlet-to-Neumann operators.
 By the Sard's theorem, it is known that almost every level set is
 regular.  Since
$0$ is assumed to be a regular value of $\phi_\lambda$, then
$N_\lambda$ is a smooth submanifold in $\mathcal{M}$ and the
boundary $\partial D_{\pm}=N_\lambda.$ By the Green formula, for any
$f\in C^\infty(\mathcal{M})$, we have
\begin{equation}
\int_{D_{\pm}} div(f\nabla \phi_\lambda)\,
dv_g=\int_{N_\lambda}<f\nabla \phi_\lambda, \nu> \,ds \label{gree}
\end{equation}
where $ds$ is the surface measure on $N_\lambda$ induced by the
metric $g$ on $\mathcal{M}$, $\nu$ is the exterior unit normal
vector on $N_\lambda$ with respect to $D_{\pm}$ respectively. Note
the Green formula is  taken on $\mathcal{M}$ with metric $g$. Since
$\phi_\lambda\equiv 0$ on $N_\lambda$, then $<\nabla \phi_\lambda,
\nu>  =\pm|\nabla \phi_\lambda|$ on $N_\lambda$. Thus, (\ref{gree})
becomes
\begin{equation}
\int_{D_{+}} div(f\nabla \phi_\lambda)\, dv_g=-\int_{N_\lambda}
f|\nabla \phi_\lambda|\,ds.\label{neg}
\end{equation}
Similarly, we have
\begin{equation}
\int_{D_{-}} div(f\nabla \phi_\lambda)\, dv_g=\int_{N_\lambda}
f|\nabla \phi_\lambda|\,ds.\label{pos}
\end{equation}
By (\ref{neg}) and (\ref{pos}), we obtain
\begin{equation}
2\int_{N_\lambda} f|\nabla \phi_\lambda|\,ds=\int_{D_{-}}
div(f\nabla \phi_\lambda)\, dv_g-\int_{D_{+}} div(f\nabla
\phi_\lambda)\, dv_g. \label{con}
\end{equation}
To obtain a lower bound of nodal sets of Steklov eigenfunctions, we
need to choose some appropriate test functions. It turns out that
$f\equiv 1$ and $f=\sqrt{1+|\nabla \phi_\lambda|^2}$ are good
choices. Let $f\equiv 1$. We are able to establish the following
proposition.
\begin{proposition}
There exists positive constant $K(\mathcal{N})$ such that, for
$\lambda>K(\mathcal{N})$,
$$\int_{N_\lambda} |\nabla
\phi_\lambda|\,ds \geq
\frac{\lambda^2}{4}\|\phi_\lambda\|_{L^{1}(\mathcal{M})}.
$$
\label{prop1}
\end{proposition}
\begin{proof}
Set $f=1$ in (\ref{con}), we have
\begin{equation}
2\int_{N_\lambda} |\nabla \phi_\lambda|\,ds=\int_{D_{-}} \triangle
\phi_\lambda\, dv_g-\int_{D_{+}} \triangle \phi_\lambda\, dv_g.
\label{fff}
\end{equation}
From lemma \ref{lem1}, we know that
$$\sqrt{-\triangle_g}=\Lambda+P_0,               $$
where $P_0\in OPS^0(\mathcal{M})$. It follows that
$$ -\triangle=\Lambda^2+P_1+P_0^2,$$
where $P_1=\Lambda P_0+P_0\Lambda \in OPS^1(\mathcal{M})$.
Therefore,
\begin{eqnarray}
\triangle \phi_\lambda &= &-(\Lambda^2+P_1+P_0^2)\phi_\lambda
\nonumber
\\
&=&-\lambda^2\phi_\lambda-P_1\phi_\lambda-P_0^2\phi_\lambda.
\nonumber
\end{eqnarray}
Substituting the above identity into (\ref{fff}) implies that

\begin{eqnarray}
2\int_{N_\lambda}|\nabla \phi_\lambda|\,ds
&=&-\int_{D_-}\lambda^2\phi_\lambda-\int_{D_-}P_1\phi_\lambda-
\int_{D_-}P^2_0\phi_\lambda \nonumber \\
&&+\int_{D_+}\lambda^2\phi_\lambda+\int_{D_+}P_1\phi_\lambda+
\int_{D_+}P^2_0\phi_\lambda \nonumber \\
&\geq&-\lambda^2\int_{D_-}\phi_\lambda+\lambda^2
\int_{D_+}\phi_\lambda-\int_{\mathcal{M}}|P_1\phi_\lambda|-\int_{\mathcal{M}}|P_0^2\phi_\lambda|
\nonumber \\
&=&\lambda^2\|\phi_\lambda\|_{L^1(\mathcal{M})}
-(1+\lambda)\|\tilde{P_0}\phi_\lambda\|_{L^1(\mathcal{M})}-\|P_0^2\phi_\lambda\|_{L^1(\mathcal{M})}
\label{fur},
\end{eqnarray}
where $\tilde{P_0}=P_1(1+\Lambda)^{-1}\in OPS^0(\mathcal{M})$. Now
there are two ``bad" terms in (\ref{fur}):
$$ \|\tilde{P_0}\phi_\lambda\|_{L^1(\mathcal{M})},
\quad \|P_0^2\phi_\lambda\|_{L^1(\mathcal{M})}.       $$
 We are able to control them by the $L^1$
norm of $\phi_\lambda$ multiplied by an $\epsilon$ power of
$\lambda$. We can establish the following lemma.
\begin{lemma}
Let $ P\in OPS^0(\mathcal{M})$. Then for any positive constant
$\epsilon$, there exists $C=C(\mathcal{N}, \epsilon)$ such that
\begin{equation}\|P\phi_\lambda\|_{L^1(\mathcal{M})}\leq C\lambda^\epsilon
\|\phi_\lambda\|_{L^1(\mathcal{M})}. \label{revi}
\end{equation}\label{lem4}
\end{lemma}
\begin{proof}
Let $\delta>0$. By H\"older's inequality,
$$\|P\phi_\lambda\|_{L^1(\mathcal{M})}\lesssim \|P\phi_\lambda\|_{L^{1+\delta}(\mathcal{M})}
\lesssim  \|\phi_\lambda\|_{L^{1+\delta}(\mathcal{M})}, $$ where we
have used lemma \ref{lem3}. As we know,
\begin{eqnarray}
\|\phi_\lambda\|_{L^{1+\delta}(\mathcal{M})}&\leq&
\|\phi_\lambda\|^{\frac{\delta}{1+\delta}}_{L^{\infty}(\mathcal{M})}
\|\phi_\lambda\|^{\frac{1}{1+\delta}}_{L^{1}(\mathcal{M})}\nonumber
\\
&=&\big(\frac{\|\phi_\lambda\|_{L^{\infty}(\mathcal{M})}}{
\|\phi_\lambda\|_{L^{1}(\mathcal{M})}}\big)^{\frac{\delta}{1+\delta}}\|\phi_\lambda\|_{L^{1}(\mathcal{M})}.
\label{est}
\end{eqnarray}
Thanks to lemma \ref{lem2}, we know
$$\|\phi_\lambda\|_{L^{\infty}(\mathcal{M})}\lesssim \lambda^{\frac{n-1}{2}}.        $$
By (\ref{less}), $$\|\phi_\lambda\|_{L^{1}(\mathcal{M})}\gtrsim
\lambda^{-\frac{n-1}{4}}.$$ Thus, from (\ref{est}), we have
$$\|\phi_\lambda\|_{L^{1+\delta}(\mathcal{M})}\lesssim \lambda^\frac{3(n-1)\delta}{4(1+\delta)}
\|\phi_\lambda\|_{L^{1}(\mathcal{M})}.$$ Selecting $\delta$ so small
that $\frac{3(n-1)\delta}{4(1+\delta)}\leq \epsilon$, then the
proposition is shown.
\end{proof}

With aid of lemma \ref{lem4}, we continue the proof of proposition
\ref{prop1}. Let's go back to (\ref{fur}). We want to control the
other two ``bad" terms in the left hand side of (\ref{fur}) by
$\lambda^2\|\phi_\lambda\|_{L^1(\mathcal{M})}$. Since (\ref{revi})
holds for any positive constant $\epsilon$, in order to achieve it,
one needs to choose  $0<\epsilon<1$. For instance, we may choose
$\epsilon=1/2$. Then we obtain
\begin{eqnarray}
2\int_{N_\lambda}|\nabla \phi_\lambda|\,ds &\geq &
\lambda^2\|\phi_\lambda\|_{L^1(\mathcal{M})}-C(1+\lambda)^{\frac{3}{2}}\|\phi_\lambda\|_{L^1(\mathcal{M})}.\nonumber
\end{eqnarray}
 Choosing $\lambda$ appropriately large which depends only on
$\mathcal{N}$,  we finally arrive at
\begin{equation}
2\int_{N_\lambda}|\nabla \phi_\lambda|\,ds \geq
\frac{\lambda^2}{2}\|\phi_\lambda\|_{L^1(\mathcal{M})}. \nonumber
\end{equation}
 We are done with the proof of proposition \ref{prop1}.
\end{proof}

Next we select the test function as $f=\sqrt{1+|\nabla
\phi_\lambda|^2}.$ We are able to prove the following proposition.
\begin{proposition}
There exists positive constant $C=C(\mathcal{N})$ such that
\begin{equation}
\int_{N_\lambda} |\nabla \phi_\lambda|^2\,ds \leq C(1+\lambda)^3
\label{upp}.
\end{equation}
\label{pro2}
\end{proposition}
\begin{proof}
Let  $f=\sqrt{1+|\nabla \phi_\lambda|^2}$ in (\ref{con}). We derive
that
\begin{eqnarray}
2\int_{N_\lambda}\sqrt{1+|\nabla \phi_\lambda|^2} |\nabla
\phi_\lambda|\,ds&=&\int_{D_{-}} div(\sqrt{1+|\nabla
\phi_\lambda|^2}\nabla \phi_\lambda)\, dv_g \nonumber
\\&&-\int_{D_{+}} div(\sqrt{1+|\nabla \phi_\lambda|^2}\nabla
\phi_\lambda)\, dv_g \nonumber
\\
&\leq & \int_{\mathcal{M}}| div(\sqrt{1+|\nabla
\phi_\lambda|^2}\nabla \phi_\lambda)|\, d v_g \nonumber \\
&\lesssim& \int_{\mathcal{M}}(1+|\nabla
\phi_\lambda|^2)^{-1/2}|\nabla^2\phi_\lambda||\nabla
\phi_\lambda|^2\,dv_g
\nonumber \\
&&+\int_{\mathcal{M}}(1+|\nabla \phi_\lambda|^2)^{1/2}|\triangle
\phi_\lambda| \, dv_g. \nonumber
\end{eqnarray}
Furthermore, we get
\begin{eqnarray}
\int_{N_\lambda}|\nabla \phi_\lambda|^2 \,ds &\lesssim&
\int_{\mathcal{M}}(1+|\nabla
\phi_\lambda|^2)^{1/2}|\nabla^2\phi_\lambda|\,dv_g\nonumber \\
&\lesssim & \big(\int_{\mathcal{M}}(1+|\nabla \phi_\lambda|^2 \,
dv_g\big)^{\frac{1}{2}}\big(\int_{\mathcal{M}}|\nabla^2\phi_\lambda|^2\,dv_g\big)^{\frac{1}{2}}
\nonumber \\
 &\lesssim &(1+\lambda)^3,
\end{eqnarray}
where lemma \ref{lem3} has been used in last inequality.
\end{proof}

We are ready to give the proof of Theorem  \ref{th1}. We use an idea
in \cite{HS} by Hezari and Sogge.
\begin{proof}[Proof of Theorem \ref{th1}]
On one hand, by proposition \ref{pro2},
\begin{eqnarray}
\int_{N_\lambda} |\nabla \phi_\lambda|\,ds &\leq & (\int_{N_\lambda}
|\nabla \phi_\lambda|^2\,ds)^{\frac{1}{2}}|N_\lambda|^{\frac{1}{2}}
\nonumber \\
 &\lesssim& \lambda^{\frac{3}{2}} |N_\lambda|^{\frac{1}{2}}.
 \label{las}
\end{eqnarray}
On the other hand, from proposition \ref{prop1}, we have
\begin{eqnarray}
\int_{N_\lambda} |\nabla \phi_\lambda|\,ds &\geq
&\frac{\lambda^2}{4}\|\phi_\lambda\|_{L^1(\mathcal{M})} \nonumber \\
&\gtrsim & \lambda^{2-\frac{n-1}{4}}, \label{lass}
\end{eqnarray}
where we have used (\ref{less}) in last inequality. Combining the
estimates (\ref{las}) and (\ref{lass}), we arrive at
$$|N_\lambda| \gtrsim  \lambda^{\frac{3-n}{2}}.   $$
\end{proof}

\end{document}